\documentclass{article}
\usepackage[utf8]{inputenc}
\usepackage{amsmath,amsthm,amssymb}
\usepackage{complexity}
\usepackage{dsfont}
\usepackage[colorlinks=true, allcolors=blue,backref]{hyperref}
\usepackage[noabbrev,capitalize,nameinlink]{cleveref}
\usepackage[T1]{fontenc}
\usepackage{tikz}
\usetikzlibrary{calc}
\usetikzlibrary{math}
\usepackage{algorithm}
\usepackage[noend]{algpseudocode} 
\usepackage{comment}
\usepackage{enumitem}

\theoremstyle{plain}
\newtheorem{theorem}{Theorem}[section]
\newtheorem{lemma}[theorem]{Lemma}

\newtheorem{proposition}[theorem]{Proposition}

\newtheorem{conjecture}[theorem]{Conjecture}
\newtheorem{problem}[theorem]{Problem}

\DeclareMathOperator*{\Bin}{Bin}
\let\Pr\relax
\DeclareMathOperator*{\Pr}{\mathbb{P}}
\let\Ex\relax
\DeclareMathOperator*{\Ex}{\mathbb{E}}
\DeclareMathOperator*{\Var}{Var}

\title{Set mappings for general graphs}

\author{
Lior Gishboliner\thanks{Department of Mathematics, University of Toronto, Canada.
Research supported in part by the NSERC Discovery Grant ``Problems in Extremal and Probabilistic Combinatorics".
\emph{Email}: \href{mailto:lior.gishboliner@utoronto.ca}{\tt lior.gishboliner@utoronto.ca}.}
\and Zhihan Jin\thanks{Department of Mathematics, ETH, Z\"urich, Switzerland. 
Research supported in part by SNSF grant 200021-228014.
\emph{Email}: \href{mailto:zhihan.jin@math.ethz.ch}{\tt zhihan.jin@math.ethz.ch}, \href{mailto:benjamin.sudakov@math.ethz.ch}{\tt benjamin.sudakov@math.ethz.ch}.}
\and Benny Sudakov\footnotemark[2]
}

\date{}
\begin{document}
\maketitle 

\begin{abstract}
The study of extremal problems for set mappings has a long history. It was introduced in 1958 by Erd\H{o}s and Hajnal, who considered the case of cliques in graphs and hypergraphs. Recently, Caro, Patk\'os, Tuza and Vizer revisited this subject, and initiated the systematic study of set mapping problems for general graphs. In this paper, we prove the following result, which
answers one of their questions. Let 
$G$ be a graph with $m$ edges and no isolated vertices and let $f : E(K_N) \rightarrow E(K_N)$ such that $f(e)$ is disjoint from $e$ for all $e \in E(K_N)$. Then for some absolute constant $C$, as long as $N \geq  C m$, there is a copy $G^*$ of $G$ in $K_N$ such that $f(e)$ is disjoint from $V(G^*)$ for all $e \in E(G^*)$. The bound $N = O(m)$ is tight for cliques and is tight up to a logarithmic factor for all $G$. 
\end{abstract}

\section{Introduction}\label{sec:intro}
Combinatorial problems for set mappings were introduced in 1958 by Erd\H{o}s and Hajnal \cite{EH}, who raised the following problem.
For fixed integers $k,\ell$ and a growing integer $N$, consider all functions $f : \binom{[N]}{k} \rightarrow \binom{[N]}{\ell}$ with the property that $f(X) \cap X = \emptyset$ for every $X \in \binom{[N]}{k}$. Let $p(f)$ be the maximum size of a set $P \subseteq [N]$ such that $f(X) \cap P = \emptyset$ for every $X \in \binom{P}{k}$. Let $p_{k,\ell}(N)$ be the minimum of $p(f)$ over all such functions $f$. 
%Note that for $\ell < \ell'$ we have $p_{k,\ell}(N) \geq p_{k,\ell'}(N)$. Indeed, given a function $f : \binom{[N]}{k} \rightarrow \binom{[N]}{\ell}$ with $f(X) \cap X = \emptyset$ for all $X$, we can define a function $f' : \binom{[N]}{k} \rightarrow \binom{[N]}{\ell'}$ by adding to $f(X)$ any set of $\ell'-\ell$ vertices disjoint from $X$ to obtain $f'(X)$. It is easy to see that $p(f) \geq p(f')$, implying the claim. 
%
Erd\H{o}s and Hajnal \cite{EH} proved that 
$\Omega\big(N^{\frac{1}{k+1}}\big) \leq p_{k,\ell}(N) \leq O\big( (N\log N)^{\frac{1}{k}} \big)$, where the upper bound is established by considering a random function $f$. Spencer \cite{Spencer}, using a now textbook application of the probabilistic method, improved the lower bound to $p_{k,\ell}(N) \geq \Omega\big(N^{\frac{1}{k}}\big)$, matching the upper bound up to a logarithmic term. Later, Conlon, Fox and Sudakov \cite{CFS} removed the logarithmic term in the case that $\ell \geq (k-1)!$, proving that $p_{k,\ell}(N) = \Theta\big (N^{\frac{1}{k}} \big)$ for such $\ell$. (The case $k=2$ of this result was proved earlier by F\"uredi \cite{Furedi}.)
Throughout the years, there have been several other works studying problems of a similar type, see \cite{AlonCaro,ACT,AM,Caro,Caro_Schoenheim}. 

Recently, Caro, Patk\'os, Tuza and Vizer \cite{CPTV_Ramsey,CPTV_Turan} returned to this topic and initiated the study of set mapping problems for general graphs. 
In \cite{CPTV_Ramsey}, they introduced the following parameter. For a graph $G$, let $w(G)$ be the minimum $N$ such that for every function $f : E(K_N) \rightarrow E(K_N)$ satisfying $f(e) \cap e = \emptyset$ for every $e \in E(K_N)$, there is a copy $G^*$ of $G$ in $K_N$ such that $f(e) \cap V(G^*) = \emptyset$ for every $e \in E(G^*)$. 
Here and throughout, we use $K_N$ to denote the complete graph on $N$ vertices and assume that $G$ has no isolated vertices. 
Note that for $G = K_n$, we are asking for a set $P$ of size $n$ such that $f(e) \cap P = \emptyset$ for every $e \in \binom{P}{2}$. Hence, this case recovers the parameter $p_{2,2}(N)$ introduced above.
Namely, the problem of estimating $w(K_n)$ is equivalent to the problem of estimating $p_{2,2}(N)$, and the aforementioned results of Spencer and Conlon-Fox-Sudakov give that $w(K_n) = \Theta(n^2)$. By taking a random mapping $f$, Caro, Patk\'os, Tuza and Vizer \cite{CPTV_Ramsey} showed that for every graph $G$ with $m$ edges, it holds that $w(G) \geq \Omega(\frac{m}{\log m})$. They further asked to obtain sharp upper bounds on $w(G)$. Answering this question, we use probabilistic arguments to prove the essentially tight bound $w(G) \leq O(e(G))$ for all graphs $G$. 

\begin{theorem}\label{thm:main}
There is an absolute constant $C > 0$ such that for every graph $G$ with $m$ edges and no isolated vertices, it holds that $w(G) \leq Cm$. 
\end{theorem}

\noindent
By the aforementioned lower bound of \cite{CPTV_Ramsey}, Theorem \ref{thm:main} is tight up to a factor of $O(\log m)$. Furthermore, it is tight when $G$ is a clique and when $|E(G)| = O(|V(G)|)$ (because trivially $w(G) \geq |V(G)|$). 

The proof of Theorem \ref{thm:main} is given in the following section. Section \ref{sec:concluding} contains some additional observations and open problems regarding the parameter $w(G)$ and related parameters.

\section{Proof of Theorem \ref{thm:main}}\label{sec:proof}
In this section, we prove the following theorem, which immediately implies \cref{thm:main}.
% An {\em embedding} of a graph $G$ in $K_N$ is simply an injection $\varphi : V(G) \rightarrow V(K_N)$.
We use the word {\em embedding} to mean injection.
\begin{theorem}\label{theorem: technical}
    There exists an absolute constant $C>0$ such that the following holds.
    Let $m, \ell$ be integers and let $N \ge C\ell m$.
    Let $G$ be a graph with $m$ edges and no isolated vertices and let $f:\binom{[N]}{2} \to \binom{[N]}{\ell}$ be a function such that $f(e)\cap e = \emptyset$ for every $e \in \binom{[N]}{2}$.
    Then there exists an embedding $\varphi:V(G) \to [N]$ such that $f\big(\varphi(v)\varphi(u)) \cap \varphi(V(G)) = \emptyset$ for every edge $vu \in E(G)$.
\end{theorem}
The case $\ell = 2$ is Theorem \ref{thm:main}.
Also, we remark that similar to \cite{CPTV_Ramsey}, by taking a random mapping $f$, it is easy to see that $N=\Omega(\ell m/\log(\ell m))$ is necessary for the conclusion of Theorem \ref{theorem: technical} to hold.

The rest of this section is devoted to proving Theorem \ref{theorem: technical}. Let $n$ be the number of vertices of $G$. 
% So $n \le 2m$ because $G$ has no isolated vertices. 
We assume that $m$ is sufficiently large. For small $m$, we can simply increase the absolute constant $C$. 
Throughout this proof, the $o(1)$ notation is always with respect to $m \rightarrow \infty$. 
Also, by adding edges and vertices, if necessary, we may assume that $m$ is a perfect square.
This can be achieved by adding $O(\sqrt{m})$ edges and vertices. 
% Hence, we can make sure less than $n/2$ vertices in $G$ are isolated. 
% Note that this can be achieved by adding at most $2\sqrt{m}$ edges and at most $2\sqrt{m} + 2\frac{n}{\sqrt{m}}$ vertices. 
% This is feasible because the number of edges we added is of lower order.
% Hence, $n < 2(m+1)$. {\color{red} Fix this paragraph.}
Then, by adding isolated vertices, we can assume that $n/\sqrt{m}$ is a power of 2.
This requires at most doubling the number of vertices of the current graph. Hence, in the resulting graph (after the changes), at most half of the vertices are isolated.
With a slight abuse of notation, we still denote the resulting graph by $G$ and its number of vertices and edges by $n$ and $m$, respectively.
So, $n \le 4m$ (because at least half of the vertices are not isolated).

Let $u_1,\dots,u_n$ be an ordering of the vertices of $G$ with non-increasing degrees, i.e. $d(u_1) \geq \dots \geq d(u_n)$. 
We denote this order by $\prec$; namely, $u_i \prec u_j$ for every $1 \leq i < j \leq n$.
Write $T:=\log_2(\frac{n}{\sqrt{m}}) \in \mathbb{N}$. 
Let $U_0 = \{u_i : 1 \leq i \leq \sqrt{m}\}$, and for each 
$1 \leq j \leq T$, let 
$U_j = \{u_i : \sqrt{m} \cdot 2^{j-1} < i \leq \sqrt{m} \cdot 2^j\}$. 
Note that $|U_j| = 2^{j-1}\sqrt{m}$ for every $1 \leq j \leq T$, so 
$|U_0| + \dots + |U_{j-1}| = |U_j|$. Also, $V(G) = U_0 \cup U_1 \cup \dots \cup U_T$.

Fix any function $f:\binom{[N]}{2}\to\binom{[N]}{\ell}$ such that $f(e) \cap e = \emptyset$ for every $e \in \binom{[N]}{2}$.
Let $X\subseteq [N]$ be the set of vertices $x \in [N]$ such that $x \in f(e)$ for at most $\ell N$ choices of $e \in \binom{[N]}{2}$.
By double-counting the pairs $(x,e) \in [N] \times \binom{[N]}{2}$ with $x \in f(e)$, we see that $(N-|X|)\ell N \le \binom{N}{2}\ell$.
Hence, $|X| > N/2 = (C/2) \ell m$.
From now on, our goal is to find an embedding $\varphi:V(G)\to X$ such that $f\big(\varphi(v)\varphi(u)) \cap \varphi(V(G)) = \emptyset$ for every edge $vu \in E(G)$.

To this end, we will find disjoint subsets $X_0,\dots,X_T \subseteq X$ with certain good properties (which will be elaborated later) and apply the following (deterministic) embedding algorithm with parameter $t=T$.
Here and throughout, for a sequence of sets $A_0,\dots, A_T$, we write $A_{\le j}$ for $A_0\cup\dots\cup A_j$ whenever $j \in \{0,\dots,T\}$.
\begin{algorithm}[H] \caption{Embedding algorithm} \label{alg:embedding}
    
    \textbf{Input:} $t \in \{0,\dots,T\}$ and disjoint subsets $X_0,\dots,X_t \subseteq X$. \\
    \textbf{Output:} An embedding $\varphi:U_{\le t} \to X_{\le t}$ with $\varphi(U_j) \subseteq X_j$ for all $0 \le j \le t$, or report failure.
    
    \medskip

    Go over $j=0,\dots,t$ in order and for each $j$, go over all vertices in $U_j$ in the ordering $\prec$.
    For each vertex $u \in U_j$, let $X_u$ be the set of all vertices $x \in X_j$ satisfying the following:
    \begin{enumerate}[label=(\alph*),ref=(\alph*)]
        \item\label{item: injection} $x \neq \varphi(v)$ for every $v \prec u$. 
        \item\label{item: later vertices} $x \not\in f\big(\varphi(v)\varphi(w)\big)$ for every edge $vw \in E(G)$ with $v,w \in U_0 \cup \dots \cup U_{j-1}$.  
        \item\label{item: former vertices} For each neighbour $v$ of $u$ (in $G$) with $v \prec u$, $f\big(\varphi(v)x\big) \cap X_{\le j} = \emptyset$.
    \end{enumerate}
    If $X_u$ is empty, we stop and report failure. 
    Otherwise, set $\varphi(u)$ to be the smallest element in $X_u$ with respect to the natural ordering on $X\subseteq V(K_N) = [N]$.
\end{algorithm}
% The idea is that after embedding, every edge $vu \in E(G)$ (say $v\prec u$ and $u \in U_j$) will forbid all vertices $x \in f\big(\varphi(u)\varphi(v)\big)$, i.e. no vertices $w \in V(G)$ can be embedded to any such $x$.
% There are two possibilities: either $x \in U_{j+1}\cup\dots\cup U_T$ or $x \in U_0\cup\dots\cup U_j$, and we will deal with them separately.
% For the former case,  we do not allow vertices in $U_{j+1}\cup\dots\cup U_{T}$ to be embedded to $x$, which is guaranteed by \Cref{item: former vertices}.
% We will show that given the choices of $X_0,\dots,X_j$, the randomness of $X_{j+1},\dots,X_{T}$ guarantees that most choices of $x$ satisfies \Cref{item: former vertices}.
% This corresponds to property \ref{item: later constraints}.
% On the other hand, for the latter case, we simply do not allow $u$ to be embedded to this $\varphi(u)$, which is guaranteed by \Cref{item: later vertices}
% As we will see later, if we sample $X_0,\dots,X_T$ properly, there are ``few'' bad  pairs $(x,y) \in X_{\le j}$ with $f(xy) \cap X_{\le j}$, so we can always find ``many'' candidate $x \in X_j$ for $X_u$ satisfying \Cref{item: later vertices}.
% This corresponds to property \ref{item: former bad pairs}.
Observe that if the above algorithm is successful, then the 
resulting embedding $\varphi:U_{\le t} \to X_{\le t}$ satisfies that for every edge $vw \in E(G)$ with $v,w \in U_{\le t}$, no $x \in f\big(\varphi(u)\varphi(v)\big)$ belongs to $\mathrm{Im}(\varphi)$. Indeed, 
suppose that $v\prec w$ and $w \in U_j$.
There are two possibilities: either $x \in X_0\cup\dots\cup X_j$ or $x \in X_{j+1}\cup\dots\cup X_t$. If $x \in X_0\cup\dots\cup X_j$, then, when we embed $w$, we make sure in \Cref{item: former vertices} that no vertices in $f(\varphi(v)\varphi(w))$ belong to $X_{\leq j}$. And if $x \in X_{j+1}\cup\dots\cup X_t$, then \Cref{item: later vertices} forbids $x$ from being the image of any vertex. Thus, if the above algorithm is successful when applied with $t = T$, then it generates an embedding $\varphi$ as required by Theorem \ref{theorem: technical}.
Note also that \Cref{alg:embedding} is deterministic, so we can speak of its output given input $X_0,\dots,X_t$ (i.e., this is well-defined). 

In order for \Cref{alg:embedding} to work, we first find disjoint subsets $X_0',\dots,X_T' \subseteq X$ 
with the following properties (later on we will carefully choose $X_j \subseteq X_j'$ for $j\in\{0,\dots,T\}$ as the input to the algorithm).
\begin{enumerate}[label=(\arabic*),ref=(\arabic*)]
    \item\label{item: size} $3.9|U_j| < |X_j'| < 4.1 |U_j|$ for all $j \in \{0,\dots,T\}$.
    \item\label{item: later constraints} For each $j \in \{0,\dots,T-1\}$ and every choice of subsets $X_i \subseteq X_i'$ for $i\in\{0,\dots,j\}$, let $\varphi:U_{\le j}\to X_{\le j}$ be the output of \Cref{alg:embedding} with input $X_0,\dots,X_{j}$ (no requirement if the algorithm fails). Then, the set $L=\bigcup_{vu} f\big(\varphi(v)\varphi(u)\big)$, where $vu$ enumerates over all edges in $G$ with $v,u\in U_{\le j}$, satisfies that $|L \cap X_{j+1}'| \le |X_{j+1}'|/9$.
    \footnote{We find property \ref{item: later constraints} to be the most surprising aspect of our analysis. Namely, the sets $X'_0,\dots,X'_T$ that will be used to find a good input for the algorithm are defined in terms of the outputs of the algorithm.}
    \item\label{item: former bad pairs} For every $0 \le i \le j \le T$, let $r_{i,j}$ be the number pairs of distinct vertices $(x,y) \in X_i'\times X_j'$ such that $f(xy)\cap X'_{\le j}\neq \emptyset$. Then $r_{i,j} \le {|U_i||U_j|^2}\big/{5m}$. 
\end{enumerate}

Let us briefly comment on these properties and their relation to \Cref{alg:embedding}. Property \ref{item: later constraints} will guarantee that in \Cref{item: later vertices} of the algorithm, only few vertices of $X_j$ are discarded due to edges $vw$ with $v,w \in U_{\leq j-1}$. And Property \ref{item: former bad pairs} will guarantee that in \Cref{item: former vertices} of the algorithm, only few vertices $x$ are discarded on account of $f(\varphi(v)x)$ belonging to $X_{\leq j}$ for some previously embedded $v$. 
Note that $i=j$ is allowed in Property \ref{item: former bad pairs}.

We now show the existence of the desired embedding given sets $X_0',\dots,X_T'$ satisfying \Cref{item: size,item: later constraints,item: former bad pairs}.
\begin{lemma}\label{lemma: given conditions}
    Suppose that disjoint subsets $X_0',\dots,X_T' \subseteq X$ satisfy \Cref{item: size,item: later constraints,item: former bad pairs}.
    Then there are subsets $X_j \subseteq X'_j$, $0 \leq j \leq T$, such that applying 
    \Cref{alg:embedding} with input $X_0,\dots,X_T$ returns an embedding $\varphi: V(G)\to X$ satisfying that $f\big(\varphi(v)\varphi(u)\big) \cap \varphi(V(G)) = \emptyset$ for every edge $vu \in E(G)$.
\end{lemma}
\begin{proof}
    For every $j \in \{0,\dots,T\}$, write $d_j := \max_{u \in U_j} d(u)$. The choice of the ordering $\prec$ and the sets $U_0,\dots,U_T$ guarantees that $d_0\ge d_1\ge \dots \ge d_T$, and that 
    $d(u) \ge d_j$ for every $j \in \{1,\dots,T\}$ and $u \in U_{j-1}$.
    Also, recall that less than $n/2$ vertices of $G$ are isolated and $|U_T|=|U_0|+\dots+|U_{T-1}| = \frac{n}{2}$.
    Hence, $d_T > 0$.
    We have 
    \begin{equation} \label{eq: weighted summation of degrees}
        2m = \sum_{u \in V(G)} d(u)
        \ge \sum_{j=1}^T d_j|U_{j-1}|
        \ge \frac{1}{2}\sum_{j=1}^T d_j|U_j|.
    \end{equation}
    The last inequality uses that $|U_{j-1}| = \frac{1}{2}|U_j|$ for $2 \leq j \leq T$, and $|U_1| = |U_0|$.
    
    For $0 \le i \le j \le T$ and $x \in X_i'$, let $r_j(x)$ denote the number of $y \in X'_j\setminus \{x\}$ such that $f(xy) \cap X'_{\le j} \neq \emptyset$.
    Note that $\sum_{x \in X_i'} r_j(x) = r_{i,j}$.
    For $0 \le i \le j \le T$ with $(i,j) \neq (0,0)$, let $B_{i,j}$ be the set of vertices $x \in X_i'$ such that $r_j(x) \ge |U_j|/d_j$ (these are ``bad'' vertices in $X_i'$ with respect to $X_j'$).
    Then, by property \ref{item: former bad pairs}, 
    \begin{equation} \nonumber
        |B_{i,j}| \le \frac{r_{i,j}}{|U_j|/d_j}
        \le \frac{|U_i|U_j|^2/5m}{|U_j|/d_j}
        = |U_i|\cdot \frac{|U_j|d_j}{5m}.
    \end{equation}
    For $i=j=0$, let $B_{0,0}$ be the set of vertices $x \in X_0'$ with $r_0(x) \ge 1$ (these are ``bad'' vertices in $X_0'$ with respect to $X_0'$).
    Using property \ref{item: former bad pairs} and that $|U_0| = \sqrt{m}$, we obtain
    \begin{equation} \nonumber
        |B_{0,0}| \le r_{0,0} \le \frac{|U_0|^3}{5m}
        = \frac{1}{5}|U_0|.
    \end{equation}
    Fixing $i \in \{1,\dots,T\}$ and summing over all $j \ge i$, we get that 
    \begin{equation} \nonumber
        \sum_{j=i}^T |B_{i,j}|
        \le |U_i|\cdot \frac{\sum_{j=i}^T|U_j|d_j}{5m}
        \le |U_i|\cdot \frac{4m}{5m}
        \le |U_i|.
    \end{equation}
    Here, we used \cref{eq: weighted summation of degrees} in the second inequality.
    Similarly, for $i=0$, we get that 
    \begin{equation} \nonumber
        \sum_{j=0}^T |B_{0,j}| = |B_{0,0}| + \sum_{j=1}^T |B_{0,j}|
        \le \frac{1}{5}|U_0| + |U_0|\cdot  \frac{\sum_{j=1}^T|U_j|d_j}{5m}
        \le \frac{1}{5}|U_0|+ |U_0|\cdot \frac{4m}{5m}
        \le |U_i|.
    \end{equation}
    
    Now, put $X_i:= X_i'\setminus\bigcup_{j=i}^T B_{i,j}$ for each $i \in \{0,\dots,T\}$.
    By the above discussion, we know that $|X_i| \ge |X_i'| - |U_i|$.
    We now prove that \Cref{alg:embedding} with input $X_0,\dots,X_T$ will return an embedding $\varphi: U_{\le T} \to X_{\le T}$ (i.e., the algorithm will not end in failure).
    It suffices to show that for every $j \in \{0,\dots,T\}$ and every $u \in U_j$, we have $X_u \neq \emptyset$, where $X_u$ is defined in \Cref{alg:embedding}.
    We prove this by induction on $j$.
    When $j = 0$, the definitions of $B_{0,0}$ and $X_0$ imply that $r_0(x) = 0$ for every $x \in X_0$. 
    This means that $f(xy) \cap X_0 = \emptyset$ for all distinct $x,y \in X'_0$.
    Hence, for every $u \in U_0$ we have $X_u = X_0 \setminus \{\varphi(v) : v \prec u\}$ (because \Cref{item: later vertices,item: former vertices} do not rule out any vertices). 
    We have $|\{\varphi(v) : v \prec u\}| < |U_0|$ and $|X_0| \geq |X_0'| - |U_0| \geq |U_0|$, using property \ref{item: size}. 
    Hence, $X_u$ is always non-empty, and the algorithm does not fail for any vertex $u \in U_0$.
	
    When $j \ge 1$, suppose that the algorithm successfully embedded $U_{\le j-1}$. 
    Let $L=\bigcup_{vw} f\big(\varphi(v)\varphi(w)\big)$ where $vw$ runs over all $vw \in E(G)$ with $v,w\in U_{\le j-1}$.
    Property \ref{item: later constraints} guarantees $|L \cap X_{j}'| \le |X_{j}'|/9$.
    So, for every $u \in U_j$, \Cref{item: later vertices} rules out at most $|L \cap X_{j}'| \le |X_{j}'|/9$ vertices in $X_j$.
    Then, we bound the number of vertices ruled out by \Cref{item: former vertices}. 
    By the definitions of the sets $X_i$ and $B_{i,j}$, any $v \prec u$ satisfies that $r_j(\varphi(v)) < |U_j|/d_j$.
    Recall that $r_j(\varphi(v))$ is an upper bound for the number of $x \in X_j$ satisfying $f(\varphi(v)x) \cap X_{\le j} \neq \emptyset$.
    Hence, for any given neighbour $v$ of $u$ (in $G$) with $v \prec u$, less than $|U_j|/d_j$ choices $x \in X_j$ are ruled out by \Cref{item: former vertices} on account of having $f(\varphi(v)x) \cap X_{\le j} \neq \emptyset$.
    Considering all neighbours $v$ of $u$ with $v \prec u$, \Cref{item: former vertices} rules out less than $|U_j|/d_j \cdot d(u)\le |U_j|$ vertices (recalling that $d(u) \le d_j$ by the definition of $d_j$).
    Finally, \Cref{item: injection} rules out less than $|U_j|$ vertices.
    Hence, it follows that $$|X_u| \ge |X_j| - |U_j| - |L\cap X_j'| - |U_j| \ge |X_j'|-3|U_j|-|X_j'|/9> 0,$$
    where we used $|X_j| \geq |X'_j| - |U_j|$ in the second inequality and property \ref{item: size} in the last one.
    This means that the algorithm does not fail for vertex $u$, and therefore, the algorithm will successfully return an embedding $\varphi: U_{\le T} \to X_{\le T}$.

    We now complete the proof of Lemma \ref{lemma: given conditions} by showing that $\varphi$ satisfies that $f\big(\varphi(v)\varphi(u)) \cap \varphi(V(G)) = \emptyset$ for every $vu \in E(G)$.
    % First, Item (a) guarantees that $\varphi$ is injective.
    Indeed, let $vu \in E(G)$ such that $v \prec u$ and take any $w \in V(G) \setminus \{v,u\}$.
    Suppose $u \in U_j$.
    If $w \in U_{j+1}\cup\dots\cup U_T$, then when embedding $w$ in \Cref{alg:embedding}, \Cref{item: later vertices} guarantees that $\varphi(w) \notin f\big(\varphi(v)\varphi(u)\big)$.
    If $w \in U_0\cup\dots\cup U_j$, then when embedding $u$, \Cref{item: former vertices} guarantees that $f\big(\varphi(v)\varphi(u)\big) \cap X_{\le j} = \emptyset$, so $\varphi(w) \notin f\big(\varphi(v)\varphi(u)\big)$, as $\varphi(w) \in X_{\leq j}$.
    This proves the lemma.
\end{proof}

We are left to find disjoint sets $X_0',\dots,X_T' \subseteq X$ satisfying properties \ref{item: size} - \ref{item: former bad pairs}.
\begin{lemma}\label{lemma: exist good sets}
    There exist disjoint sets $X_0',\dots,X_T' \subseteq X$ satisfying properties \ref{item: size} - \ref{item: former bad pairs}.
\end{lemma}
\begin{proof}
    For $j \in \{0,\dots,T\}$, set $$\alpha_j := 4|U_j|/|X|.$$
    Note that $\alpha_0+\dots+\alpha_T = \frac{4}{|X|}(|U_0|+\dots+|U_T|) = \frac{4n}{|X|} \leq \frac{1}{2}$ as $|X| \ge N/2=(C/2)\ell m$ and $n \le 4m$.
    We sample disjoint sets $X_0',\dots,X_T' \subseteq X$ as follows: for each $x \in X$, take a random $i_x \in \{\star,0,\dots,T\}$ such that $\Pr[i_x = i] = \alpha_i$ for $0 \leq i \leq T$, and $\mathbb{P}[i_x = \star] = 1-\sum_{i=0}^T\alpha_i$ (independently of all other vertices). 
    If $i_x \neq \star$ then place $x$ in $X_i'$ for $i = i_x$.
    The sets $X_0',\dots,X_T'$ are clearly disjoint.
    It suffices to prove that with positive probability, $X_0',\dots,X_T'$ satisfy properties \ref{item: size} - \ref{item: former bad pairs}.

    Property \ref{item: size} follows directly from the Chernoff bound.
    Indeed, for $j \in \{0,\dots,T\}$, we have 
    $|X'_j| \sim \Bin(|X|,\alpha_j)$, so $\Ex[|X_j'|] = 4|U_j|$.
    By the Chernoff bound, the probability that $3.9|U_j| \leq |X_j'| \leq 4.1|U_j|$ is at least $1-\exp(-\Omega(|U_j|)) \ge 1-\exp(\Omega(\sqrt{m}))$ since $|U_j| \ge |U_0| = \sqrt{m}$.
    By a union bound over all $j \in \{0,\dots,T\}$, we know that $(1)$ holds with probability at least $1-T\exp(-\Omega(\sqrt{m})) = 1 - o(1)$.
    Here, we used that $T = O(\log m)$.
    % and that $m$ is sufficiently large.
    % From now on, we will assume $3.9|U_j| < |X_j| < 4.1 |U_j|$ for all $j \in \{0,\dots,T\}$.

    For property \ref{item: later constraints}, fix $j \in \{0,\dots,T-1\}$ and condition on the choice of $X_0',\dots,X_{j}'$. We also assume that 
    $3.9|U_i| \leq |X'_i| \leq 4.1 |U_i|$ for all $0 \leq i \leq j$, which holds with probability $1 - o(1)$ by the above. 
    Put $X' := X\setminus X_{\le j}'$.
    Conditioned on $X_0',\dots,X_j'$, the set $X_{j+1}'$ is distributed as a random subset of $X'$ where each $x \in X'$ is chosen with probability $\frac{\alpha_{j+1}}{1-\sum_{i\le j}\alpha_i} \le 2\alpha_{j+1}$.
    %This inequality holds because $\sum_{i\le j}\alpha_i \le \sum_{i< T}\alpha_i=\alpha_T \le 1/2$.
    Fix any $X_i \subseteq X'_i$ for $0 \leq i \leq j$. The number of choices for $X_0,\dots,X_j$ is 
    $$
        2^{\sum_{i=0}^j |X_i'|}
        \le 2^{4.1\sum_{i=0}^j |U_i|}
        = 2^{4.1|U_{j+1}|}.
    $$
    Now, having fixed $X_0,\dots,X_j$, let $\varphi:U_{\le j} \to X_{\le j}$ be the output of \Cref{alg:embedding} with input $X_0,\dots,X_j$.
    We may assume that the algorithm succeeds to produce such a $\varphi$, since otherwise there is no requirement with respect to $X_0,\dots,X_j$. 
    Write $L:=\bigcup_{vu} f\big(\varphi(v)\varphi(u)\big)$ where $vu \in E(G)$ and $v,u\in U_{\le j}$.
    Note that $L$ is uniquely determined by $X_0,\dots,X_j$ via \Cref{alg:embedding} and that $|L| \le m\ell$.
    Hence, $|L\cap X_{j+1}'|$ is stochastically dominated by a binomial distribution with $m\ell$ trials and success probability $2\alpha_{j+1}$.
    This implies that 
    \begin{equation} \nonumber
      \begin{aligned}
        &\,\Pr\!\left[ \Big|L\cap X_{j+1}'\Big| > \frac{\alpha_{j+1}|X|}{10} \right]
        \le \binom{m\ell}{{\alpha_{j+1}|X|}/{10}} \big(2\alpha_{j+1}\big)^{{\alpha_{j+1}|X|}/{10}}
        \le \left(\frac{20em\ell}{|X|}\right)^{\alpha_{j+1}|X|/10} \\
        \leq&\, \left(\frac{40em\ell}{N}\right)^{\alpha_{j+1}|X|/10}
        \leq \left(\frac{40em\ell}{C\ell m}\right)^{(2/5)|U_{j+1}|}
        \le 2^{-5|U_{j+1}|}.
      \end{aligned}
    \end{equation}
    Here, we used that $|X| \ge N/2 \geq (C/2)\ell m$ where $C>0$ is a sufficiently large constant.
    Taking the union bound over all choices of $X_0,\dots,X_j$, we see that the probability that $|L\cap X_{j+1}'| \le \alpha_{j+1}|X|/10$ for all such choices is at least $1-2^{4.1|U_{j+1}|}\cdot 2^{-5|U_{j+1}|}=1-2^{-\Omega(\sqrt{m})}$.
    Moreover, we already know that with probability $1 - o(1)$ it holds that $|X_{j+1}'| \ge 3.9|U_{j+1}|=(39/40)\alpha_{j+1}|X|$, so $|L\cap X_{j+1}'| \le \alpha_{j+1}|X|/10$ implies $|L\cap X_{j+1}'| \le |X_{j+1}'|/9$.
    Taking the union bound over all $j \in \{0,\dots,T-1\}$, we see that with probability $1 - o(1)$, properties \ref{item: size} and \ref{item: later constraints} both hold. 

    We are left to show that \ref{item: former bad pairs} holds with probability at least $1/4$, say.
    Fix $0 \le i \le j \le T$.
    For every pair of distinct $x,y \in X$ and every $z \in f(x,y) \cap X$, let $I_{x,y,z}$ be the indicator of the event that $x \in X_i', y \in X_j', z\in X_{\le j}'$.
    Clearly, $r_{i,j} \le r_{i,j}^\star:=\sum_{x\neq y,z\in f(x,y)\cap X} I_{x,y,z}$ (since the same pair $x\not =y$ might be counted for several $z$).
    For convenience, put $\alpha_{\le j} :=\sum_{k=0}^j \alpha_k = 2\alpha_j$, using that $|U_j|=\sum_{k=0}^{j-1} |U_{k}|$.
    Writing $M:=|X| \ge N/2\ge (C/2)\ell m$ (where $C$ is a sufficiently large constant), we have 
    \begin{equation}
    \label{eq1}
        \Ex[r^\star_{i,j}] \le M(M-1)\ell\cdot \alpha_i\alpha_j\alpha_{\le j}
        \le 2M^2\ell \alpha_i\alpha_j^2
        = \frac{128 \ell |U_i||U_j|^2}{M}
        < \frac{|U_i||U_j|^2}{m}.
    \end{equation}

    Next, we bound the variance of $r^\star_{i,j}$.
    It suffices to consider the covariance of two events $I_{x,y,z}$ and $I_{x',y',z'}$ where $x\neq y,z \in f(xy)\cap X$ and $x'\neq y',z' \in f(x'y')\cap X$. We split these according to the size of the intersection 
    $|\{x,y,z\}\cap\{x',y',z'\}|$.
    \begin{itemize}
        \item If $\{x,y,z\}\cap\{x',y',z'\}=\emptyset$, then $I_{x,y,z}$ and $I_{x',y',z'}$ are independent, so their covariance is 0.
        \item If $|\{x,y,z\}\cap\{x',y',z'\}|=1$, then the probability that $I_{x,y,z}=I_{x',y',z'}=1$ is at most $\alpha_i\alpha_j^2\alpha_{\le j}^2=4\alpha_i\alpha_j^4$, using that $\alpha_i \leq \alpha_j \leq \alpha_{\leq j}$.
        Let us count the pairs $(x,y,z),(x',y',z')$ with $|\{x,y,z\} \cap \{x',y',z'\}| = 1$. 
        The number of choices for $\{x,y,z\}$ is at most $|X|^2\ell = M^2\ell$. 
        Fixing $x,y,z$, we know that one of the vertices $x',y',z'$ belongs to $\{x,y,z\}$. 
        If $x' \in \{x,y,z\}$, then there are 3 choices for $x'$, at most $M = |X|$ choices for $y'$ and at most $\ell$ choices for $z'$. The same holds if $y' \in \{x,y,z\}$.
        And if $z' \in \{x,y,z\}$ then there are $3$ choices for $z'$ and then at most $2\ell N$ choices for $x'y'$, by the definition of the set $X$ (and as $z' \in X$). 
        In total, the number of choices for $(x,y,z),(x',y',z')$ with $|\{x,y,z\} \cap \{x',y',z'\}| = 1$ is at most $M^2\ell \cdot \left( 6 M\ell + 6 \ell N \right) \leq M^2\ell \cdot 18M\ell = 18M^3\ell^2$.
    \item 
    If $|\{x,y,z\}\cap\{x,y,z\}|=2$, then the probability that $I_{x,y,z}=I_{x',y',z'}=1$ is at most $\alpha_i\alpha_j\alpha_{\le j}^2=4\alpha_i\alpha_j^3$.
    Also, the number of such pairs is at most $13M^3\ell$.
    Indeed, there are at most $M^2 \ell$ choices for $(x,y,z)$. Fixing $x,y,z$, two of the vertices $x',y',z'$ must be in $\{x,y,z\}$. If these two vertices are $x',z'$ then there are at most $6$ choices for $x',z'$ and then at most $M$ choices for $y'$. The same applies to the pair $y',z'$. And if $x',y' \in \{x,y,z\}$ then there are $6$ choices for $x',y'$ and at most $\ell$ choices for $z'$. This gives at most $M^2\ell \cdot (12M + 6\ell) \leq 13M^3\ell$ choices in total, as claimed. 
    \item If $|\{x,y,z\}\cap\{x,y,z\}|=3$, then the probability that $I_{x,y,z}=I_{x',y',z'}=1$ is at most $\alpha_i\alpha_j\alpha_{\le j}=2\alpha_i\alpha_j^2$. Also, there are at most $5M^2\ell$ such choices for $(x,y,z),(x',y',z')$. (I.e., $x',y',z'$ is any permutation of $x,y,z$ except $x,y,z$ itself).
    \end{itemize}
    
    % Recall that there are at most $\ell N\le 2\ell M$ choices $e \in \binom{X}{2}$ with $z \in f(e)$ (by our choice of $X$).
    % Hence, the number of pairs of $(x,y,z)$ and $(x',y',z')$ with $|\{x,y,z\}\cap\{x,y,z\}|=1$ is at most $M^2\ell\cdot 12M\ell=12 M^3\ell^2$.
    \noindent
    These observations, together with (\ref{eq1}), plus the fact that $M\alpha_j  = 4|U_j| \ge 1$, imply that 
    \begin{equation} \nonumber
      \begin{aligned}
        \Var[r^\star_{i,j}]
        &\le \Ex[r^\star_{i,j}] + \sum_{(x,y,z)\neq (x',y',z')} \mathrm{Cov}(I_{x,y,z},I_{x',y',z'}) \\ &\leq 
        2M^2 \ell \alpha_i \alpha_j^2 + 18M^3\ell^2 \cdot 4 \alpha_i \alpha_j^4 + 
        13M^3 \ell \cdot 4\alpha_i\alpha_j^3 + 
        5M^2 \ell \cdot 2\alpha_i \alpha_j^2
        \\ &= 
        72 M^3\ell^2\alpha_i \alpha_j^4 +
        52 M^3 \ell\alpha_i \alpha_j^3 + 
        12M^2 \ell \alpha_i \alpha_j^2
        \\ &< 
        100M^3\ell^2\alpha_i\alpha_j^4 + 
        100 M^3\ell\alpha_i\alpha_j^3.
        % 12M^3\ell^2\cdot 4\alpha_i\alpha_j^4 + (12M^3\ell+6M^2\ell^2)\cdot 4\alpha_i\alpha_j^3 + 5M^2\ell\cdot 2\alpha_i\alpha_j^2 \\
        % &\le 12M^2\ell\alpha_i\alpha_j^2 + 48M^3\ell^2\alpha_i\alpha_j^4+48M^3\ell\alpha_i\alpha_j^3+24M^2\ell^2\alpha_i\alpha_j^3 \\
        % &< 100 M^3\ell\alpha_i\alpha_j^3 + 100M^3\ell^2\alpha_i\alpha_j^4.
      \end{aligned}
    \end{equation}
    We now plug in $\alpha_i = 4|U_i|/M,\alpha_j=4|U_j|/M$, and use that $|U_j|\le |U_T| = n/2 \le 2m$ and that $M\ge (C/2)\ell m$ with $C > 0$ sufficiently large. We get:  
    \begin{equation} \nonumber
      \begin{aligned}
        \Var[r^\star_{i,j}]
        \le 
        100 M^3 \ell^2 \frac{4^5 |U_i||U_j|^4}{M^5} + 
        100 M^3 \ell \frac{4^4 |U_i||U_j|^3}{M^4}  
        \le  
        \frac{100 \cdot 4^6 \cdot |U_i||U_j|^4}{C^2m^2} + 
        \frac{100 \cdot 2^9 \cdot |U_i||U_j|^3}{Cm}  
        < \frac{|U_i||U_j|^3}{m}.
      \end{aligned}
    \end{equation}
    By Chebyshev's inequality, we have that 
    \begin{equation} \nonumber
	\Pr\!\left[ r_{i,j}^\star \geq \frac{5|U_i||U_j|^2}{m} \right] 
        \leq \Pr\!\left[ \Big| r_{i,j}^\star - \Ex[r_{i,j}^\star] \Big| \geq \frac{4|U_i||U_j|^2}{m}\right] 
        \leq  \frac{\Var[r_{i,j}^\star]}{16|U_i|^2|U_j|^4/m^2} 
        < \frac{m}{16|U_i||U_j|}.
    \end{equation}
    By the union bound, the probability that there are $0 \leq i \leq j \leq T$ with $r_{i,j}^\star \geq \frac{5|U_i||U_j|^2}{m}$ is at most 
    $$
        \sum_{0 \leq i \leq j \leq T}\frac{m}{16|U_i||U_j|} 
        \le \frac{1}{16} \left(\sum_{i=0}^T\frac{\sqrt{m}}{|U_i|}\right)^2 
        \le \frac{1}{16} \left( 1+\sum_{i\ge 0} 2^{-i} \right)^2
        = \frac{9}{16}.
    $$
    Recall that $r_{i,j} \le r^\star_{i,j}$.
    Therefore, property \ref{item: former bad pairs} holds with probability at least $7/16>1/4$, and this completes the proof of Lemma \ref{lemma: exist good sets} and hence of Theorem \ref{theorem: technical}.
\end{proof}

\section{Concluding remarks and open problems}\label{sec:concluding}
By combining Theorem \ref{thm:main} and the lower bound from \cite{CPTV_Ramsey}, we get that for every graph $G$ with $m$ edges and no isolated vertices, it holds that $\Omega(\frac{m}{\log m}) \leq w(G) \leq O(m)$. Are there graphs $G$ for which the lower bound is the answer? In particular, we wonder if this is the case for $G = K_{n,n}$.
\begin{problem}
    Determine $w(K_{n,n})$.
\end{problem}

The definition of $w(G)$ generalizes naturally to $k$-uniform hypergraphs ($k$-graphs for short). 
Write $K_N^{(k)}$ for the $N$-vertex complete $k$-graph.
For a $k$-graph $G$, let $w(G)$ be the minimum $N$ such that for every map $f : E(K_N^{(k)}) \rightarrow E(K_N^{(k)})$ satisfying $f(e) \cap e = \emptyset$ for every $e \in E(K_N^{(k)})$, there is a copy $G^*$ of $G$ in $K_N^{(k)}$ such that $f(e) \cap V(G^*) = \emptyset$ for every $e \in E(G^*)$. 
%For $G = K_n^{(k)}$, the result of Spencer \cite{Spencer} mentioned in the introduction implies that $w(K_n^{(k)}) = O(n^k)$. 
It is natural to conjecture that the bound $w(G) = O(e(G))$ continues to hold for hypergraphs (extending Theorem \ref{thm:main}), though our methods do not easily generalize.  
\begin{conjecture}
    For every $k$-graph $G$ with $m$ edges and no isolated vertices, it holds that $w(G) = O(m)$. 
\end{conjecture}

We can show that the conjecture holds if $G$ is almost regular, in the sense that its maximum degree $\Delta$ and average degree $d$ satisfy $\Delta = O(d)$. This follows from the following somewhat standard local lemma argument. 
% We need the asymmetric form of the local lemma, which we now recall.
% \begin{theorem}[Lov\'asz local lemma, see Chapter 5 in
% \cite{AlonSpencer}]
% Let $B_1,\dots,B_m$ be events in a probability space. For each $i \in [m]$, let $\Gamma_i \subseteq [m] \setminus \{i\}$ such that $B_i$ is independent of the events $\left( B_j : j \in [m] \setminus (\Gamma_i \cup \{i\}) \right)$. Suppose that there are $x_1,\dots,x_m \in [0,1)$ such that for every $i \in [m]$,
% $$
% \mathbb{P}[B_i] \leq x_i \cdot \prod_{j \in \Gamma_i} (1-x_j).
% $$
% Then there is an outcome for which none of $B_1,\dots,B_m$ happen.
% \end{theorem}
\begin{theorem}[Lov\'asz local lemma, see Chapter 5 in
\cite{AlonSpencer}]\label{thm:LLL}
Let $B_1,\dots,B_m$ be events in a probability space. Suppose that there are $p,d$ such that $ep(d+1) \leq 1$, $\mathbb{P}[B_i] \leq p$ for each $i \in [m]$, and each $B_i$ is independent from all but at most $d$ other events $B_j$. 
Then there is an outcome for which none of $B_1,\dots,B_m$ happen.
\end{theorem}
\begin{proposition}\label{prop:regular}
Let $G$ be a $k$-graph with $n$ vertices and maximum degree $\Delta$. Then $w(G) \leq 10k^2 n\Delta$.
\end{proposition}
\begin{proof}
    Suppose that $V(G) = [n]$. Set $N = 10k^2 n\Delta$. 
    Let a function $f : E(K_N^{(k)}) \rightarrow E(K_N^{(k)})$ be given such that $f(e) \cap e = \emptyset$ for every $e \in E(K_N^{(k)})$.
    Sample vertices $x_1,\dots,x_n \in [N]$ uniformly at random. We define two types of bad events:
    \begin{itemize}
    \item For $1 \leq i < j \leq n$, let $A_{i,j}$ be the event that $x_i = x_j$.
    \item For an edge $e \in E(G)$ and $i \in [n] \setminus e$, let $B_{e,i}$ be the event that $x_i \in f(\{x_j : j \in e\})$.
\end{itemize}
We have $\mathbb{P}[A_{i,j}] = \frac{1}{N}$ and 
$\mathbb{P}[B_{e,i}] = \frac{k}{N}$. Let us now consider the dependencies. Each event $A_{i,j}$ depends only on variables $x_i,x_j$, and each event $B_{e,i}$ depends only on the variables $x_i$ and $(x_j : j \in e)$. Two events are independent if their sets of variables are disjoint. Fix any set of vertices $X \subseteq [n]$ of size $|X| \leq k+1$. The number of events $A_{i,j}$ which have a variable in $X$ is at most $|X|n \leq (k+1)n$. Let us now bound the number of events $B_{e,i}$ with a variable in $X$. In the case that $i \in X$, the number of options is at most $|X| \cdot e(G) \leq (k+1)\frac{\Delta n}{k} \leq 2\Delta n$. In the case that $e$ has a vertex in $X$, the number of options is at most $|X| \cdot \Delta \cdot n \leq (k+1)\Delta n$. Altogether, we see that each bad event ($A_{i,j}$ or $B_{e,i}$) depends on at most $(k+3)\Delta n \leq 3k\Delta n$ other bad events. By Theorem \ref{thm:LLL}, applied with $p =\frac{k}{N}$ and $d = 3k\Delta n$, there is an outcome where no bad event occurs. If this happens, $x_1,\dots,x_n$ form a copy of $G$ (with $x_i$ playing the role of $i$) such that $f(e) \cap \{x_1,\dots,x_n\} = \emptyset$ for every edge $e$ of the copy.  
\end{proof}

Caro \cite{Caro} and later Conlon, Fox and Sudakov \cite{CFS} considered a variant of the original set mapping problem (defined in Section \ref{sec:intro}) where instead of requiring that $f(X)$ is disjoint from $P$ for every $X \in \binom{P}{k}$, one only requires that $f(X)$ is not contained in $P$. Moreover, one can replace the assumption that $f(X) \cap X = \emptyset$ for every $X$ with an upper bound on $|f(X) \cap X|$. In the context of general (2-uniform) graphs, there are two natural parameters of this type, which we now define. As in \cite{CPTV_Ramsey}, for $a \in \{0,1\}$, let $F_{N,a}$ denote the set of functions $f : E(K_N) \rightarrow E(K_N)$ with $|f(e) \cap e| \leq a$ for every $e \in E(K_N)$. For a graph $G$, let $g_a(G)$ be the minimum $N$ such that for every function $f \in F_{N,a}$, there is a copy $G^*$ of $G$ in $K_N$ satisfying that $f(e) \notin E(G^*)$ for every $e \in E(G^*)$. Such a copy of $G$ is called {\em $f$-free}. Trivially, $g_1(G) \geq g_0(G) \geq |V(G)|$. The standard first moment argument gives the following lower bounds on $g_a(G)$:

\begin{proposition}\label{prop:f-free lower bounds}
    Let $G$ be a graph with $n$ vertices and $m$ edges. Then  
    $g_1(G) \geq g_0(G) \ge \Omega \big( \frac{m}{\sqrt{n \log n}} + n\big)$ and
    $g_1(G) \geq \Omega \big(\frac{1}{n\log n}\sum_{u \in V(G)}d(u)^2 + n\big) \geq \Omega\big( \frac{m^2}{n^2\log n} + n \big)$.
\end{proposition}
\begin{proof}
    We may assume that $m \geq 3n$.
    Otherwise, both bounds clearly hold since $g_0(G),g_1(G) \geq n$. 
    To prove the bounds, we will sample certain random mappings $f \in F_{N,a}$.
    
    Consider first $g_0(G)$. 
    Sample $f$ by choosing, for each $e \in E(K_N)$, a random edge $e'$ with $|e \cap e'| = 0$, uniformly and independently, and setting $f(e) = e'$. 
    Fix any copy $G^*$ of $G$ in $K_N$.
    For each edge $e \in E(G^*)$, there are at least $m-2n$ edges in $G^*$ that are disjoint from $e$, so the probability that $f(e) \in E(G^*)$ is at least $(m-2n)/\binom{N-2}{2}=\Theta(m/N^2)$.
    By the independence of $f(e)$ for all $e \in E(G^*)$, the probability that $G^*$ is $f$-free is at most $\big(1-(m-2n)/\binom{N-2}{2}\big)^{m}=e^{-\Theta(m^2/N^2)}$.
    Then, if $N = o\big(\frac{m}{\sqrt{n \log n}}\big)$, the union bound over all copies $G^*$ of $G$ implies that the probability that there exists an $f$-free copy of $G$ is at most $N^ne^{-\Theta(m^2/N^2)} = o(1)$.
    Here, we also used $\log N = O(\log n)$.
    This means $g_0(G) = \Omega\big(\frac{m}{\sqrt{n \log n}}\big)$.
    Also, $g_0(G) \ge n$, so $g_0(G) \ge \Omega \big( \frac{m}{\sqrt{n \log n}} + n\big)$.

    For $g_1(G)$, we proceed similarly, but now we choose $f(e)$ to be a random edge $e'$ that satisfies $|e \cap e'| = 1$. 
    For each vertex $u \in V(G)$, write $d(u)$ for the degree in $G$.
    Fix a copy $G^*$ of $G$ in $K_N$.
    For any edge $e \in G^*$, say corresponding to an edge $uv$ in $G$, there are $d(u)+d(v)-2$ edges $e' \in E(G^*)\setminus \{e\}$ that are incident to $e$.
    Therefore, the probability that $G^*$ is $f$-free is at most $$
        \prod_{uv \in E(G)} \bigg(1-\frac{d(u)+d(v)-2}{2N-4}\bigg)
        \le \exp\bigg(-\!\!\!\sum_{uv \in E(G)} \frac{d(u)+d(v)-2}{2N-4}\bigg)
        \le \exp\bigg(-\frac{1}{2N-4}\sum_{u \in V(G)} d(u)^2 + \frac{2m}{2N-4}\bigg).
    $$
    Observe that by Jensen's inequality and our assumption that $m \ge 3n$, we have $$\sum_{u \in V(G)}d(u)^2 \geq n \cdot \left(\frac{2m}{n} \right)^2 = \frac{4m^2}{n} \ge 12m.$$
    So, the above probability is at most $\exp\!\big(\!-\frac{1}{4N-8}\sum_{u \in V(G)} d(u)^2\big)$.
    Then, if $N = o\big(\frac{1}{n\log n}\sum_{u\in V(G)}d(u)^2\big)$, the probability that there exists an $f$-free copy of $G$ is at most $$
        N^n \cdot \exp\bigg(-\frac{1}{4N-8}\sum_{u \in V(G)} d(u)^2\bigg)
        = N^n \cdot e^{-\omega(n\log n)} = o(1).
    $$
    This shows that $g_1(G) = \Omega\big(\frac{1}{n\log n}\sum_{u\in V(G)}d(u)^2\big)$.
    In addition, $g_1(G) \ge n$ so $$g_1(G) = \Omega\Big(\frac{1}{n\log n}\sum_{u\in V(G)}d(u)^2+n\Big).$$
    Finally, recall that as $\sum_{u \in V(G)} d(u)^2 \ge 4m^2/n$, we have $g_1(G) = \Omega\big(\frac{m^2}{n^2\log n}+n\big)$.
\end{proof}

\noindent
As for upper bounds, the following bound on $g_1(G)$ (and hence on $g_0(G)$) follow from Theorem \ref{theorem: technical}.
\begin{theorem}\label{thm:g_1}
    Let $G$ be a graph with $m$ edges and no isolated vertices. Then $g_1(G) = O(m)$.
\end{theorem}
\begin{proof}
    Set $N = Cm$ for a large enough $C$, and let $f : E(K_N) \rightarrow E(K_N)$ with $|f(e) \cap e| \leq 1$ for every $e \in E(K_N)$. Define a function $f' : E(K_N) \rightarrow [N]$ by setting $f'(e)$ to be a vertex in $f(e) \setminus e$ for every $e \in E(K_N)$. By Theorem \ref{theorem: technical}, there is a copy of $G^*$ with $f'(e) \cap V(G^*) = \emptyset$ for every $e \in E(G^*)$. This guarantees that $f(e) \notin E(G^*)$. Hence, $g_1(G) \leq N$.  
\end{proof}
It was shown by Conlon, Fox and Sudakov \cite[Theorem 2.2(i)]{CFS} that the bound in Theorem \ref{thm:g_1} is tight for $G = K_n$. Also, it is tight up to a logarithmic factor for $G = K_{k,n-k}$ by Proposition \ref{prop:f-free lower bounds}, because $m = k(n-k)$ and $\sum_{u \in V(G)}d(u)^2 \geq k(n-k)^2 = \Omega(n m)$ (here we assume that $k \leq \frac{n}{2}$).

It would be interesting to decide if the lower bounds in Proposition \ref{prop:f-free lower bounds} are tight, up to the logarithmic factors. In particular, we conjecture the following:
\begin{conjecture}\label{conj:g_0}
    Let $G$ be a graph with $n$ vertices, $m$ edges and no isolated vertices. Then $g_0(G) = O\big( \frac{m}{\sqrt{n}} + n \big)$.
\end{conjecture}
\noindent
The bound in Conjecture \ref{conj:g_0} is tight for $G = K_n$; see \cite[Theorem 2.2(ii)]{CFS}. Also, using the (asymmetric) local lemma similarly to the proof of Proposition \ref{prop:regular}, one can show that $g_0(G) = O(\sqrt{\Delta m} + n)$, which implies Conjecture \ref{conj:g_0} if $\Delta = O(m/n)$, i.e. if the maximum degree is at most a constant times the average degree.

\bibliographystyle{abbrv}
\bibliography{library}

\end{document}